\documentclass[10pt]{article}
\usepackage{hyperref}
\usepackage[english]{babel}
\usepackage[latin1]{inputenc}
\usepackage{amsmath,amssymb}
\usepackage{latexsym}
\usepackage{latexsym}
\usepackage{amsfonts}
\usepackage{eucal}
\usepackage{amstext}
\usepackage{amsthm}
\usepackage{amssymb}
\usepackage{amscd}
\theoremstyle{definition}

\theoremstyle{plain}
\newtheorem{thm}{Theorem}[section]
\newtheorem{prop}[thm]{Proposition}
\newtheorem{lem}[thm]{Lemma}
\newtheorem{corol}[thm]{Corollary}
\newtheorem{obser}[thm]{Remark}

\def\prx{\partial_x}
\def\pry{\partial_y}

\newcommand{\ns}[2]{ \left\| #1 \right\|_{#2}}

\newcommand{\sgn}{\mathrm{sgn}}
\newcommand{\normp}[2]{\Vert#1\Vert_{#2}}

\newcommand{\SSh}{\mathcal{S}}
\newcommand{\re}{\mathbb{R}}
\newcommand{\R}{\mathbb{R}}

\newcommand{\nat}{\mathbb{N}}

\title{On the Cauchy problem of a two-dimesional Benjamin-Ono equation.}
\author{Germán Preciado López and F\'elix H. Soriano M\'endez}
\linethickness{.12mm}
\begin{document}\maketitle
\begin{abstract}
In this work we shall show that the Cauchy problem  
\begin{equation}
\left\{
\begin{aligned}
&(u_t+u^pu_x+\mathcal H\partial_x^2u+ \alpha\mathcal H\partial_y^2u )_x - \gamma u_{yy}=0 \quad p\in{\nat} \\
 &u(0;x,y)=\phi{(x,y)}
 \end{aligned}
 \right.
\end{equation} 
is locally well-posed in the Sobolev spaces $H^s({\re}^2)$, $X^s$ and
weighted spaces $X_s(w^2)$, for $s>2$.
\end{abstract}
\section{Introduction} 
The purpose of this work is to show that the Cauchy problem 
\begin{eqnarray}
  (u_t+u^pu_x+\mathcal H\partial_x^2u+ \alpha\mathcal H\partial_y^2u
  )_x  -\gamma u_{yy}=0 \label{BOeq},
\end{eqnarray}
is locally well-posed in the Sobolev spaces $H^s({\re}^2)$ and $X^s$,
and in the weighted spaces $X_s(w^2)$, for $s>2$ (see the Section
\ref{sec2} for the notations used here). We also prove global well-posedness
for small enough initial data and examine the asymptotic behaviour of
the solutions for these initial datas.\par It should be noted that the
equation \eqref{BOeq} is the model of dispersive long wave motion in a
weakly nonlinear two-fluid system, where the interface is subject to
capillarity and bottom fluid is infinitely deep (see \cite{ablow},
\cite{abloseg} and \cite{kim}). For this equation, with $\alpha=0$,
the local well-posedness was proven in \cite{guoboling}. Also, the
existence of solitary wave solution was proved in \cite{presor} (for
the case $\alpha=0$ in \cite{spahani} it was provided an incomplete
proof).\par
Observe that \eqref{BOeq} is a two-dimensional case of the
Benjamin-Ono equation
\begin{eqnarray}\partial_tu+\mathcal H\partial_x^2{u}+u\partial_xu=0,\label{BOeq1}
\end{eqnarray}
which describes certain models in physics about wave propagation in a
stratified thin regions (see \cite{Benjamin} and \cite{Ono}). This last
equation shares with the equation KdV 
\begin{eqnarray} 
u_t+u_x+uu_x+u_{xxx}=0 \label{KdVeq}
\end{eqnarray}
many interesting properties. For example, they both have infinite
conservation laws, they have solitary waves as solutions which are
stable and behave like soliton (this last is evidenced by the
existence of multisoliton type solutions) (see \cite{ablow} and
\cite{matsuno}). Also, the local and global well-posedness was proven
in the Sobolev spaces context (in low regularity spaces inclusive,
see, e.g., \cite{iorio2}, \cite{ponce}, \cite{kenko}, \cite{kochtz}
and \cite{tao})\par
The plan of this paper is the following. In Section \ref{sec2} we present
the basic notations and results that we will need. In Section
\ref{sec3} we examine the local well-posedness in $H^s$ and $X^s$. To
do so, we will use the abstract theory developed by Kato in
\cite{katoLN448} (see also \cite{kato1983}) to prove the local
well-posedness of quasi-linear equations of evolution.  Kato
considered the problem
\begin{equation}\label{Q}
\begin{array}{ll}
&\partial_t u +A(t,u)u=f(t,u)\in X, \ 0< t , \\
&u(0)=u_0\in Y,
\end{array}
\end{equation}
in a Banach space $X$ with inicial data in a dense subspace $Y$ of
$X$, where $A$ is a map from $\re\times X$ into the linear operators
of $X$ with dense domain and $f(t,u)$ is a function from $\re \times
Y$ to $X$, which satisfy the following conditions:
\par
$(X)$ There exists an isometric isomorphism  $S$ from $Y$ to $X$. \\
There exist $T_0>0$ and $W$ a open ball with center $w_0$ such that: \par
$(A_1)$ For each $(t,y)\in [0,T_0]\times W$, the linear operator
$A(t,y)$ belongs to $G(X,1,\beta)$, where $\beta$ is a positive
real number. In other words, $-A(t,y)$ generates a $C_0$
semigrup such that
\[
\normp{e^{-sA(t,y)}}{\mathcal B(X)}\le e^{\beta s}, \ \text{para}\
s\in[0,\infty).
\]
it should be noted that if  $X$ is a Hilbert space, $A \in
G(X,1,\beta)$ if, and only if,
\begin{enumerate}
\item[a)] $\langle Ay, y\rangle_{X}\ge -\beta\normp{y}{X}^2$ for all
$y\in D(A)$,
\item[b)] $(A+\lambda)$ is onto for all $\lambda >\beta.$
\end{enumerate} (See \cite{katoPT} or \cite{reed}) \par
$(A_2)$ For all $(t,y)\in[0,T_0]\times W$ the operator
$B(t,y)=[S,A(t,y)]S^{-1}\in \mathcal B(X)$ and is is uniformly
bounded, i.e., there exists $\lambda_1>0$ such that
\begin{align*}
& \normp{B(t,y)}{\mathcal B(X)}\le\lambda_1\ \ \text{for all}\
(t,y)\in[0,T_0]\times W, \\ \intertext{In addition, for some $\mu_1>0$,
    it hat, for all $y$ and $z\in W$,}
&\normp{B(t,y)-B(t,z)}{\mathcal B(X)}\le\mu_1\normp{y-z}{Y}.
\end{align*}\par
$(A_3)$ $Y\subseteq D(A(t,y))$, for each $(t,y)\in[0,T_0]\times W,$
(the restriction of $A(t,y)$ to $Y$ belongs to $\mathcal B(Y,X)$) and,
for each fixed $y \in W$, $t\to A(t,y)$ is strongly continuous.
Furthermore, for each fixed $t\in[0,T_0]$, it is satisfied the following
Lipschitz condition,
\[
\normp{A(t,y)-A(t,z)}{\mathcal B(Y,X)}\le\mu_2\normp{y-z}{X},
\]
where $\mu_2\ge 0$ is a constant.

$(A_4)$ $A(t,y)w_0\in Y$ for all $(t,y)\in[0,T]\times
W$. Also, there exists a constant $\lambda_2$ such that
\[
\normp{A(t,y)w_0}{Y}\le\lambda_2, \ \text{for all}\
(t,y)\in[0,T_0]\times W
\]\par
$(f_1)$ $f$ is a bounded function from $[0,T_0]\times W$ in $Y$, i.e.,
there exists $\lambda_3$ such that
\[
\normp{f(t,y)}{Y} \le\lambda_3, \ \text{for all}\ (t,y)\in[0,T_0] \times W,
\] Besides, the function $t\in [0,T_0] \mapsto f(t,y)\in Y$ is continuous
with respect to $X$ topology and, for all $y$ and $z\in Y$, we have that
\begin{align*}
&\normp{f(t,y)-f(t,z)}{X}\le\mu_3\normp{y-z}{X}, \end{align*} when
$\mu_3\ge 0$ is a constant.
\begin{thm}[Kato]\label{tkato}
Suppose that the conditions $(X),$ $(A_1)-(A_4)$ y $(f_1)$ are
satisfied. For $u_0\in Y$, there exist $0<T<T_0$ and a unique $u\in
C([0,T];Y)\cap C^1((0,T);X)$ solution to \eqref{Q}.  Besides, the
map $u_0\to u$ is continuous in the following sence:
consider the following  sequence of Cauchy  problems,
\begin{equation}\label{dependencia}
\begin{aligned}
&\partial_t u_n+A_n(t,u_n)u_n=f_n(t,u_n) &t>0 \\
&u_n(0)=u_{n_0}  & n\in \mathbb N.
\end{aligned}
\end{equation}
Assume that conditions $(X)$, $(A_1)$--$(A_4)$ and $(f_1)$ hold
for all $n\ge 0$ in \eqref{dependencia}, with the same
$X, \ Y$ and $S$, and the corresponding $\beta$,
$\lambda_1$--$\lambda_3$, $\mu_2$--$\mu_3$ can be chosen
independently from $n$. Also assume that
\begin{align*}
\mathop{s\text{-}\lim}_{n\to\infty} A_n(t,w)&=A(t,w) \ \text{in } \ B(X,Y)\\
\mathop{s\text{-}\lim}_{n\to\infty} B_n(t,w)&=B(t,w) \ \text{in } \ B(X)\\
\lim_{n\to\infty} f_n(t,w)&=f(t,w) \ \text{in } \ Y\\
\lim_{n\to\infty} u_{n_0}&=u_0  \ \text{in } \ Y,
\end{align*}
where $s$-$\lim$ denotes the strong limit. Then, $T$ can be so chosen
in such a way that $u_n\in C([0,T],Y)\cap C^1((0,T),X)$ and
\[
\lim_{n\to\infty} \sup_{[0,T]}\normp{u_n(t)-u(t)}{Y}=0.
\]
\end{thm}
A proof of this theorem can be seen in \cite{katoLN448} and
\cite{kobayasi}.\par
In the Section \ref{sec4} is examined the local well-posedness in the
weighted spaces $X^s(w^2)$. For this, we use ideas of Milanés in
\cite{aniura2003} (see also \cite{aniuraphdt}). Milánes, in her work,
examines the local well-posedness of the problem
\begin{equation}\label{anieq}
\left\{
\begin{aligned}
&u_t+u^pu_y+\mathcal Hu_{xy}=0 \quad p\in{\nat} \\
 &u(0;x,y)=\phi{(x,y)}
 \end{aligned}
 \right.
\end{equation} 
in weighted Sobolev spaces, extending ideas developed by Iório in
\cite{iorio1986} and \cite{iorio2}. Finally, in the Section \ref{sec5}
we present the asymptotic behaviour of solutions with small initial
data. This is obtained from $L^p$-$L^q$ estimates of the group
associated to the linear part of the equation \eqref{BOeq} analogous
to those of the Schrödinger group $e^{it\Delta}$ in dimension two, as
it is done by Milánes in \cite{aniura2003} for the equation
\eqref{anieq}. Observe that this property is shared by generalized
Benjamin-Ono equation (in one dimension, see \cite{KPV1994}), from
where this result is suggested. This, also, allows to prove the global
existence for these small datas.

\section{Preliminaries}\label{sec2}
In this paper we systematically use the following notations.

\begin{enumerate}
\item $\SSh({\R}^n)$  is the Schwartz space. If $n=2$,
 we simply write $\SSh$.
\item $\SSh'({\R}^n)$ is the space of tempered distributions.  If $n=2$, we simply write $\SSh'$.
\item For $f\in{\SSh}'({\R}^n)$, $\widehat{f}$ is the Fourier
  transform of $f$ and $\check{f}$ is the inverse Fourier transform of
  $f$. We recall that $$ \widehat f(\xi) = {(2\pi)^{-\frac n2}}
  \int_{\re^n} f(x)e^{i\langle x,\xi\rangle} dx, $$ for all $\xi\in
  \re^n$, when $f\in \SSh(\re^n)$.
\item $\mathcal H=\mathcal H^{(x)}$ is the Hilbert transform with
  respect to the variable $x$. If $f\in{\SSh}({\R}^2)$, $$\mathcal H f
  (x,y) = \sqrt\frac2\pi \left(\mathrm{p.v.}  \int_{-\infty}^\infty
    \frac1{\xi-x} f(\xi,y)\, d\xi\right). $$
\item For $s\in{\R}$,
$H^s=H^s({\R}^2)$
is the Sobolev space of order $s$.  
\item The inner product in $H^s$ is denoted as $\langle f,g\rangle_s=\int_{{\R}^2}(1+\xi^2+\eta^2)^s\widehat{f}\overline{\widehat{g}}d\xi{d}\eta.$
\item $X^s= \{f\in{H}^s({\R}^2)\, \big|\, f=\partial_x g, $ for
  some $ g \in{H}^s({\R}^2) \}$.
\item $X^s(\rho)$ is the espace $X^s(\rho)=X^s\cap L^2(\rho (x,y)dxdy)$
\item $\Lambda^s=(1-\Delta)^{s/2}$.
\item $L_p^s(\re^n)=\{ f\in \SSh'(\re^n) \,\big|\, \Lambda^sf\in L_p(\re^n)\}$. 
\item For $f \in L_p^s(\re^2)$, $|f|_{p,s}= \|\Lambda^sf\|_{L_p(\re^2)}$.
\item $[A,B]$ will denote the commutator of $ A $ and $ B $.
\end{enumerate}
The following results about commutators of operators are part of the
important stock of tools that are used in the analysis.\par The first
of them is given by the following proposition due to Kato (its proof
can be found in \cite{katoLN448}).

\begin{prop}[Kato's inequality]\label{deskato}
Let $f\in H ^s$, $s>{2}$, $\Lambda=(1-\Delta^2)^{1/2}$ and
$M_f$ be the multiplication operator by $f$. Then, for $|\tilde t|,
\ |\tilde s|\le s-1$, $\Lambda^{-\tilde s}[\Lambda^{\tilde s+\tilde t
+1}, M_f]\Lambda^{-\tilde t }\in B(L^2(\mathbb R^2))$ and
\begin{equation}\label{lamilagrosa}
\left\|{\Lambda^{-\tilde s}[\Lambda^{\tilde s+\tilde t +1},
M_f]} \Lambda^{-\tilde t }\right\|_{B(L^2(\mathbb R^2))}\le
c\normp{{\nabla f} }{H^{s-1}}.
\end{equation}
\end{prop}
\begin{prop}[Kato-Ponce's inequality]\label{deskatoponce}
Let $s>{0}$, $1<p<\infty$, $\Lambda=(1-\Delta^2)^{1/2}$ and
$M_f$ be the multiplication operator by $f$. Then,  
\begin{equation}\label{lamilagrosa2}
\left|[\Lambda^{ s}, M_f] g\right |_{p}\le c\left( |\nabla
f|_{\infty} |\Lambda^{ s-1}g|_{p}+ |\Lambda^{ s}f|_{p}| g|_{\infty}
\right),
\end{equation}
for all $f$ and $g\in \SSh$
\end{prop}
\begin{corol}\label{needglob}
For $f$ and $g\in \SSh$,
$$|f,g|_{s,p}\le c\left( |f|_{\infty} |\Lambda^{ s}g|_{p}+ |\Lambda^{
    s}f|_{p}| g|_{\infty} \right).
$$
\end{corol}
The following theorem is due to A. P. Calderón (see \cite{calderon1965})

\begin{thm}[Calderón's commutator theorem]
  Let $A:{\R}\to{\R}$ be a Lipschitz function. Then, for any
  $f\in{S}({\R})$,
 \[\|[\mathcal H,A]f'\|_0\leq{C}|A'|_{\infty}\|f\|_0.\]
\end{thm}
\begin{lem} \label{gft} Let $g,h \in \mathcal{S}(\mathbb{R}^n)$ and $s
  \geq 0$. Then there exists a constant $C=C(s)$ such that
  \[ \ns{gh}{[s]} \leq C \left[ \ns{g}{A} \ns{h}{[s]}+ \ns{g}{[s]}\ns{g}{A} \right] \]
where $\ns{\phi}{[s]}=\ns{(-\Delta^2)^{\frac{s}{2}}}{o}$ y $\ns{\phi}{A}=\ns{\widehat{\phi}}{L^1}$
\end{lem}
\begin{corol} \label{Corol B2} Let $g$, $h$ and $s$ be as in the Lemma
  \ref{gft} and $\frac{n}{2} < s_0$. Then there exists a constant
  $C=C(s)$ such that \[ \ns{g \partial_x h}{s} \leq C \left(
    \ns{g}{s}\ns{h}{s} + \ns{g}{s_0}\ns{h}{s+1} \right) \]
\end{corol}
\section{Local theory in Sobolev spaces}\label{sec3}

In this section we examine the local well-posedness of a Cauchy
problem associated to a two-dimesional generalization of the
Benjamin-Ono equation given in \eqref{BOeq}.\par
First, we consider the local well-posedness in $H^s(\re^2)$ when $\gamma=0$.
\begin{thm}\label{tp}
  Let $s>2$ and $p\in{\nat}$. For $\phi\in{H^s({\R}^2)}$, there exist
  $T>0$, that depends only on $\|\phi\|_s$, and a unique
  $u\in{C([0,T],H^s({\R}^2))\cap{C}^1([0,T],H^{s-2}({\R}^2))}$
  solution to the Cauchy problem
 \begin{equation}\label{p1}
\left\{
\begin{aligned}
&u_t+\mathcal H\partial_x^2u +\alpha\mathcal H\partial_y^2u +u^pu_x=0  \\
 &u(0)=\phi .
 \end{aligned}
 \right.
\end{equation}
Furthemore, the map $\phi \mapsto u$ from $H^s$ to $C([0,T], H^s)$
is continuous.
\end{thm}
\begin{proof}
 Without loss generality we can suposse $\alpha =1$. In this
  case, $u$ is solution to \eqref{p1} if and only if $v(t)=e^{t\mathcal
    H \Delta}u(t)$ is solution to
\begin{equation}\label{p2}
\left\{
  \begin{aligned}
    &\frac{dv}{dt}+A(t,v)v=0 \\
    &v(0)=\phi,
  \end{aligned}
\right.
\end{equation}
where
\[A(t,v)=e^{t\mathcal H \Delta}(e^{-t\mathcal H \Delta}
v)^p\partial_xe^{-t\mathcal H \Delta} .\] Let us see for this
problem that each one of the conditions of the Kato's theorem
(Theorem \ref{tkato}) is satisfied. For the moment, let $X=L^2({\R}^2)$ and
$Y=H^s({\R}^2)$, for $s>2$. It is clear that $S=(1-\Delta)^\frac s2$ is
an isomorphism between $X$ and $Y$. In the following lemmas we verify
that the problem \eqref{p2} satisfies the conditions $(A_1)$--$(A_4)$
of the Theorem \ref{tkato}.
\begin{lem}\label{lema2.1}
  $A(t,v)\in G(X,1,\beta(v))$, where $\beta(v)=\frac12
  {\sup_t\|\partial_x (e^{t\mathcal H \Delta} v)^p\|_{L^\infty(\re
      ^2)}}$ (see the condition $(A_1)$ before Theorem \ref{tkato}).
\end{lem}
\begin{proof}
  Since $\{e^{-t\mathcal H \Delta}\}$ is a strongly continuous group
  of unitary operators, and thanks to the observation immediately
  below of the condition $(A_1)$ of the Theorem \ref{tkato}, it
  follows the lemma.
\end{proof}
\begin{lem}\label{lema2.2}
 If $S=(1-\Delta)^{s/2}$, then 
$$
SA(t,v)S^{-1}=A(t,v)+B(t,v),
$$
where $B(t,v)$ is a bounded operator in $L^2$, for all $t\in \re$ and
$v\in H^s$, and satisfies the inequalities
\begin{align}
 \|B(t,v)\|_{B(X)}&\leq{\lambda}(v)\label{con2}\\
\|B(t,v)-B(t,v')\|_{B(X)}&\leq{\mu}(v, v')\|v'-v\|_s\label{con3},
\end{align}
for all $t\in\re $, and every $v$ and $v'\in H^s$, where
$\lambda(v)=\sup_t C_s \| \nabla (e^{-t\mathcal H \Delta} v)^p
\|_{s-1}$ and $\mu(v,v') =C_{p,s} (\|v\|_s^{p-1}+ \|v'\|_s^{p-1}).$
\end{lem}
\begin{proof}
From the Proposition \ref{deskato} follows that
\([S,(e^{-t\mathcal H \Delta}v)^p]\partial_xS^{-1} \in{B(X)}\) and
$$
\|[S,(e^{-t\mathcal H \Delta}v)^p]\partial_xS^{-1}\|_{B(X)}
\leq{C}_s\| \nabla(e^{-t\mathcal H \Delta}v)^p\|_{s-1}.
$$
Therefore, $B(t,v)\in B(X)$ and satisfies \eqref{con2}.
\par By proceeding as above and taking into account that  
\begin{equation}\label{ineq:11}
 \| v^p -w^p
\|_{s} \le C_{p,s} (\|u\|_s^{p-1} +\| v\|_s^{p-1})\|u-v\|_s, 
\end{equation}
for all
$u$ and $v\in H^s$, we can show \eqref{con3}.
\end{proof}
\begin{lem}\label{lema2.3}
  $ H^s({\R}^2)\subset{D}(A(t,v))$ and $A(t,v)$ is a bounded operator
  from $Y={H^s({ \R}^2)}$ to $X=L^2(\re^2)$ with $$\|A(t,v)\|_ {B(
    X,Y)}\le \lambda(v), $$ for all $v\in Y$, and where $\lambda $ is
  as in the Lemma \ref{lema2.2}. Also, the function $t\mapsto
  {A}(t,v)$ is strongly continuous from $\re$ to $B(Y,X)$, for all
  $v\in{H^s}$. Moreover, the function $v\mapsto{A}(t,v)$ satisfies
  the following Lipschitz condition
\[\|A(t,v)-A(t,v')\|_{B(Y,X)}\leq\mu(v, v')\|v-v'\|_X, \]
where $\mu$ is as in the lemma above.
\end{lem}
\begin{proof}
  Inasmuch as $e^{-t\mathcal H \Delta} =(e^{t\mathcal H \Delta})^{-1}$
  is an unitary operator in $X=L^2({\R}^2)$, from the definition of
  $A(t,v)$, it follows $H^s({\R}^2)\subset{D}(A(t,v))$. In fact, $$\|
  A(t,v) f\|_0 = \|e^{-t\mathcal H \Delta}v)^p\partial_x e^{t\mathcal
    H \Delta}f\|_0 \le C_s\| (e^{-t\mathcal H
    \Delta}v)^p\|_s\|\partial_x f\|_0 \le \lambda(v)\|f\|_s, $$ for
  all $f\in Y$. \par
  Now, for all $t,t'\in{\re}$ and all $f,v \in Y$, we have
\begin{align*}
\|A(t,v)f-A(t',v)f\|_0\leq & \left\| \left( e^{t\mathcal H \Delta}-
  e^{t'\mathcal H \Delta} \right) (e^{-t\mathcal H \Delta}
v)^p\partial_x (e^{t\mathcal H \Delta}f) \right\|_0+\\
&+ \left\| \left((e^{-t\mathcal H \Delta}v)^p-(e^{-t'\mathcal H
      \Delta}v)^p \right)\partial_x(e^{t\mathcal H \Delta}f) \right\|_0\\
&+\|(e^{-t'\mathcal H \Delta}v)^p\partial_x(e^{t\mathcal H \Delta}-e^{t'\mathcal H \Delta})f\|_0
\end{align*}
Since the group $\{e^{-t\mathcal H \Delta}\}_{t \in\re}$ is strongly
continuous and the function $v \to v^p$ from $Y$ itself is continuous,
$t\mapsto {A}(t,v)$ is strongly continuous from $\re$ to $B(Y,X)$.
\par Finally, for any $t\in \re$ we have
\begin{align*}
\|A(t,v')f-A(t,v)f\|_0&\leq\|(e^{t\mathcal H \Delta}v')^p-(e^{t\mathcal H \Delta}v)^p\|_0\|\partial_xe^{t\mathcal H \Delta}f\|_{\infty}\\
&\leq{C_p}(\| (e^{t\mathcal H \Delta}v)^{p-1}\|_{\infty} + \|
(e^{t\mathcal H \Delta}v')^{p-1}\|_{\infty}) \|f\|_s \|v'-v\|_0\\
&\leq\mu(v,v')\|v'-v\|_0\|f\|_s.
\end{align*}
This completes the proof of the lemma.
\end{proof}
The immediately preceding lemmas show that the problem \eqref{p2}
satisfies the Theorem \ref{tkato}~ conditions and, therefore, for each
$\phi$ $\in H^s$, $s>2$, there exist $T>0$, which depends on
$\|\phi\|_s$, and an unique $v \in C([0,T],H^s({\R}^2)\bigcap
C^1([0,T],$ $H^{s-1}({\R}^2))$ solution to problem \eqref{p2}.  Also,
the map $\phi\mapsto {v}$ is continuous from $H^s({\R}^2)$ to
$C([0,T],H^s({\R}^2)$. Now, from the properties of group
$Q(t)=e^{-t\mathcal{H}\Delta}$ can be verified that $u(t)=Q(t)v(t)$ is
solution to \eqref{p1} and satisfies the properties enunciated in
Theorem \ref{tp}.
\end{proof}
\begin{thm}\label{tt}
  The time of existence of the solution to \eqref{p1} can be chosen
  independently from $s$ in the following sense: if
  $u\in{C}([0,T],H^s)$ is the solution to \eqref{p1} with
  $u(0)=\phi\in{H^r}$, for some $r >{s}$, then
  $u\in{C}([0,T],H^r)$. In particular, if $\phi\in{H}^{\infty}$,
  $u\in{C}([0,T],H^{\infty}).$
\end{thm}
\begin{proof}
  The proof of this result is essentially the same as part $(c)$ of
  the Theorem 1 in \cite{katoMM1979}. We will briefly outline this.  Let
  $r>s$, $u\in{C}([0,T],H^s)$ be the solution to \eqref{p1} and
  $v=e^{t\mathcal H \Delta}u$. Let us suposse that $r\le
  s+1$. Applying $\partial_x^2$ in both sides of the differential
  equation in \eqref{p2}, we arrive at the following linear evolution
  equation in $w(t)=\partial_x^2v(t)$,
\begin{equation}\label{ger}
\frac{dw}{dt}+A(t)w+B(t)w=f(t)
\end{equation}
where
\begin{align}\label{v1}
A(t)&=\partial_xe^{t\mathcal H \Delta}(u(t))^pe^{-t\mathcal H \Delta}\\
\label{ta}B(t)&=2e^{t\mathcal H \Delta}[p(u(t))^{p-1}]u_x(t)e^{-t\mathcal H \Delta}\\
\label{ja}f(t)&=-e^{t\mathcal H \Delta}[p(p-1)u^{p-2}(t)][u_x(t)]^3.
\end{align}

Since $v\in{C}([0,T),H^s)$ we have that
$w\in{C}([0,T);H^{s-2})$. Also, $w(0)=\phi_{xx}\in{H}^{r-2}$, because
$\phi\in{H}^r$. Let us prove that $w\in{C}([0,T],H^{r-2})$. To do
this, we shall prove that the Cauchy problem associated to the linear
equation lineal (\ref{ger}) is well-posed for $1-s\leq{k}\leq{s}-1.$
In this direction we have the following lemma whose proof is
completely similar to that of Lemma 3.1 in \cite{katoMM1979}.
\begin{lem}\label{un}
  The  family $\{A(t)\}_{0\leq{t}\leq{T}}$ has an unique family of
  evolution  operators $U(t,\tau)_{0\leq{t}\leq\tau\le {T}}$ in the
  spaces $X=H^h,\quad {Y}=H^k$ (in the Kato sense), where
 \begin{align}
 -s\leq{h}\leq{s}-2\quad 1-s\leq{k}\leq{s}-1\quad {k}+1\leq{h}
\end{align}
In particular, $U(t,\tau):H^r\rightarrow{H}^r$ for
$-s\leq{r}\leq{s}-1.$
\end{lem}
Then, the last lemma allows us to show that  $w$
satisfies the equation
\begin{equation}\label{vol}
w(t)=U(t,0)\phi_{xx}+\int_0^tU(t,\tau)[-B(\tau)w(\tau)+f(\tau)]d\tau.
\end{equation}
Now, since $w(0)=\phi_{xx}\in{H}^{r-2}$, by \eqref{ja}, $f$ is in
${C}([0,T],H^{s-1})\subset{C}([0,T],H^{r-2})$ (if $r\leq{s+1}$) and
${B}(t)$, given in (\ref{ta}), is a family of operators in
$\mathcal{B}(H^{r-2})$ strongly continuous for $t$ in the interval
$[0,T]$ (if $r\leq{s}+1$), from Lemma \ref{un}, the solution to
\eqref{vol} is in ${C}([0,T],H^{r-2})$ (\eqref{vol} is an integral
equation of Volterra type in $H^{r-2}$, which can be solved by
successive approximations), in others words, $\partial_x ^2u
\in{C}([0,T],H^{r-2})$. \par If $w_1(t)=\partial_{x}\partial_y v(t)$,
we have
\begin{equation}\label{ger2}
\frac{dw_1}{dt}+A(t)w_1+B_1(t)w_1=f_1(t)
\end{equation}
where
\begin{align}
\label{ta1}B_1(t)&=e^{t\mathcal H \Delta}[p(u(t))^{p-1}]u_x(t)e^{-t\mathcal H \Delta}=\frac12B(t)\\
\label{ja1}f_1(t)&=-e^{t\mathcal H
  \Delta}((p(p-1)u^{p-2}(t)[u_x(t)]^2 + p(u(t))^{p-1}u_{xx}(t))u_y(t)).
\end{align}
As above, we have
\begin{equation}\label{vol1}
w_1(t)=U(t,0)\phi_{xy}+\int_0^tU(t,\tau)[-B_1(\tau)w_1(\tau)+f_1(\tau)]d\tau.
\end{equation} 
Inasmuch as $u_{xx} \in C([0,T], H^{r-2})$, $f_1\in
{C}([0,T],H^{r-2})$. Since, also, $B_1(t)\in
\mathcal{B}(H^{r-2})$ is strongly  continuous in the interval $[0,T]$,
arguing as before, we have that $w_1\in C([0,T],
H^{r-2})$ or, equivalently, $u_{xy}\in C([0,T], H^{r-2})$\par
Analogously, if $w_2(t)=\partial_y^2 v(t)$, we have
\begin{equation}\label{ger3}
\frac{dw_2}{dt}+A(t)w_2=f_2(t)
\end{equation}
where
\begin{align}
\label{ja2}f_2(t)&=-e^{t\mathcal H
  \Delta}((p(p-1)u^{p-2}(t)u_x(t)u_y(t) + 2p(u(t))^{p-1}u_{xy}(t))u_y(t)).
\end{align}
Then,
\begin{equation}\label{vol2}
w_2(t)=U(t,0)\phi_{yy}+\int_0^tU(t,\tau)f_2(\tau)d\tau.
\end{equation} 
Since $u_{xy} \in C([0,T], H^{r-2})$, $f_2\in
{C}([0,T],H^{r-2})$. Repeating the argument above, we can conclude
that $w_1\in C([0,T], H^{r-2})$ or, equivalently, $\partial_y^2u\in
C([0,T], H^{r-2})$\par Then, we have proved that, if $s<r\le s+1$ and
$\phi \in H^r$, $u \in C([0,T], H^{r})$. To the case $r>s+1$, as
$\phi\in H^{s'}$, for $s'<r$, using a bootstrapping argument can be
shown that $u \in C([0,T], H^{r})$.
\end{proof}
Now we examine the local well-posedness of \eqref{BOeq} in
$X^s(\re^2)$ without any restriction on the parameters.
\begin{thm}\label{tp1}
  Let $s>2$ and $p\in{\nat}$. For $\phi\in{X^s({\R}^2)}$, there exist
  $T>0$, that depends only on $\|\phi\|_s$, and a unique
  $u\in{C([0,T],X^s({\R}^2))\cap{C}^1([0,T],H^{s-2}({\R}^2))}$
  solution to the Cauchy problem
 \begin{equation}\label{p3}
\left\{
\begin{aligned}
&u_t+\mathcal H\partial_x^2u +\alpha\mathcal H\partial_y^2u
-\gamma \partial^{-1}_x \partial_y^2u +u^pu_x=0  \\
 &u(0)=\phi .
 \end{aligned}
 \right.
\end{equation}
Furthemore, the map $\phi \mapsto u$ from $X^s$ to $C([0,T], X^s)$
is continuous.
\end{thm}
\begin{proof}
  The proof is basically the same as the Theorem \ref{tp}. Let
  $\mathcal A =\mathcal H\partial_x^2 +\alpha\mathcal H\partial_y^2
  -\gamma \partial^{-1}_x \partial_y^2$. It is easy to check that
  $\mathcal A$ generates a strongly continuous group in
  $H^s$. Therefore, the local well-posedness in $H^s$ of the Cauchy
  problem
 \begin{equation}\label{p4}
\left\{
  \begin{aligned}
    &\frac{dv}{dt}+A(t,v)v=0 \\
    &v(0)=\phi,
  \end{aligned}
\right.
\end{equation}
where
\[A(t,v)=e^{t\mathcal A }(e^{-t\mathcal A} v)^p\partial_xe^{-t\mathcal
  A } ,\] follows from lemmas completely analogous to the Lemmas
\ref{lema2.1}, \ref{lema2.2} and \ref{lema2.3} with which we proved
the local well-posedness of the Cauchy problem \eqref{p2}.\par Now,
let $v$ be the solution to Cauchy problem \eqref{p4} and
$u=e^{-t\mathcal A}v$. Let us prove that if $\phi\in X^s$, $u\in
C([0,T],X^s({\R}^2))$ and is solution to \eqref{p3}. From \eqref{p4}
it can be easily proved that 
\begin{equation}
\begin{aligned}\label{eq2.22}
u&=e^{-t\mathcal A}\phi+\int_0^t
e^{-(t-\tau)\mathcal A}\partial_x\left( \frac {u^{p+1}(\tau)}{p+1} \right)
d\tau\\
&=e^{-t\mathcal A}\phi+\partial_x\int_0^t
e^{-(t-\tau)\mathcal A}\left( \frac {u^{p+1}(\tau)}{p+1} \right)
d\tau.
\end{aligned}
\end{equation}
Indeed, $u\in C([0,T],H^s({\R}^2)) $ is solution to the last equation
if only if $v=e^{t\mathcal A}u$ is solution to \eqref{p4}.  Since
$H^s$ is a Banach algebra, $t\mapsto u^{p+1}(t)$ is continuous from
$[0,T]$ to $H^s$. In particular, $\int_0^t e^{-(t-\tau)\mathcal
  A}\left( {u^{p+1}(\tau)} \right) d\tau $ is a continuous function in
$t$ with values in $H^s$. Hence, if $\phi \in X^s$,
$$
\partial_x^{-1} u=e^{-t\mathcal A}\partial_x^{-1}\phi+\int_0^t
e^{-(t-\tau)\mathcal A}\left( \frac {u^{p+1}(\tau)}{p+1} \right)
d\tau\in C([0,T],H^s).
$$
Therefore $u\in C([0,T],X^s({\R}^2))$ and $u$ is solution to
\eqref{p3}.  Also, by \eqref{ineq:11}
$$
\begin{aligned}
  \sup_{t\in [0,T]} \|\partial_x^{-1} (u-\tilde u)(t)\|_s\le&
  \|\partial_x^{-1} (\phi- \tilde \phi)\|_s+ \\ &+ C_{p,s} \sup_{t\in
    [0,T]} (\|u\|_s^{p-1} +\| \tilde u\|_s^{p-1})\sup_{t\in
    [0,T]}\|(u-\tilde u)(t)\|_s,
\end{aligned}
$$
where $\tilde \phi\in X^s$ and $\tilde u\in C([0,T],X^s({\R}^2)) $ is solution to 
$$
\tilde u=e^{-t\mathcal A}\tilde\phi+ \int_0^t e^{-(t-\tau)\mathcal
  A}\partial_x\left( \frac {\tilde u^{p+1}(\tau)}{p+1} \right) d\tau.
$$ 
Therefore, the local well-posedness of \eqref{p3} is equivalent to the
local well-posedness of \eqref{p4}. This finishes the proof. 
\end{proof} 
The following theorem is totally analogous to the Theorem \ref{tt}
\begin{thm}\label{tt1}
  The time of existence of the solution to \eqref{p3} can be chosen
  independently from $s$ in the following sense: if
  $u\in{C}([0,T],X^s)$ is the solution to \eqref{p3} with
  $u(0)=\phi\in{X^r}$, for some $r >{s}$, then
  $u\in{C}([0,T],X^r)$. In particular, if $\phi\in{X}^{\infty}$,
  $u\in{C}([0,T],X^{\infty}).$
\end{thm}
\begin{proof}
  Suposse $u\in{C}([0,T],X^s)$ is the solution to \eqref{p3} with
  $u(0)=\phi\in{X^r}$ with $r>s$. To see that $u\in{C}([0,T],H^r)$, we
  repeat the same arguments that we used in the proof of Theorem
  \ref{tt}, it is just to replace the operator $\mathcal H\Delta$ with
  $\mathcal A$, the operator defined in the proof of the immediately
  above theorem. Since $$
  \partial_x^{-1} u=e^{-t\mathcal A}\partial_x^{-1}\phi+\int_0^t
  e^{-(t-\tau)\mathcal A}\left( \frac {u^{p+1}(\tau)}{p+1} \right)
  d\tau,
$$
we have that   $u\in{C}([0,T],X^{r}).$   
\end{proof}
\section{Local theory in weigthed Sobolev spaces}\label{sec4}
In this section we shall examine the local well-posedness of the
Cauchy problem \eqref{BOeq} in some weigthed Sobolev spaces. We use
ideas developed in \cite{iorio1986}, \cite{iorio2} and
\cite{aniura2003}.\par
First, we consider the case $\gamma=0$.
\begin{thm}\label{teorpeso1}
  Assume that $w$ is a weight with its first and second derivatives
  bounded and, for some $\lambda^*$, there exist $C_{\lambda}>0$ such that
\[|w(x,y)|\leq{C}_{\lambda}{e}^{\lambda(x^2+y^2)}, \]
for all $(x,y)\in{\R}^2$ and all  $\lambda\in(0,\lambda^*)$. Let 
$${X}^s(w^2)= \{ f\in X^s \, |\, wf\in L^2 \}.$$ This is a Hilbert space
with the inner product $\langle \cdot, \cdot \rangle_{w,s} =\langle
\cdot, \cdot \rangle_{X^s} +\langle \cdot, \cdot \rangle_{L^2(w^2)}$.
Then, for $s>2$, the Cauchy problem \eqref{p1} is local well-posed in
$X^s(w^2)$.
\end{thm}
\begin{proof} In this  proof we use following lemma.
\begin{lem}
  For $w$ as in the theorem. Let $w_{\lambda}(x,y)=w(x,y)
  e^{-{\lambda}(x^2+y^2)}$. there exist constants $c_1,\ c_2,\
  c_3$ and $c_4$ independient of $\lambda$ and such that
 \[|\nabla{w}_{\lambda}|_{\infty}\leq{c}_1|\nabla{w}|_{\infty}+c_2\] and
 \[|D^{\alpha}w_{\lambda}|_{\infty}\leq{c}_3|\nabla{w}|_{\infty}+
 |D^{\alpha}w|_{ \infty}+c_4, \]
 for any  multindex $\alpha=(\alpha_1,\alpha_2)$ with $|\alpha|=2$.
\end{lem}
In view of the local well posedness in $X^s$, it is enough with
examining some estimates of $L^2(w^2)$ norm. Well, with this purpose
let $w_{\lambda}(x,y)=w(x,y)e^{-{\lambda}(x^2+y^2)}.$
It is clear that $||w_{\lambda}u(t)||_0<{\infty}$ and $||w_{\lambda}
u_t(t)||_0< {\infty}$, for all ${t}\in[0,T]$ and all $\lambda
>0$. Hence, multiplying on both sides of the equation \eqref{p1}
by $w_{\lambda}^2 u$ and integrating we obtain
\[\frac{1}{2}{\frac{d}{dt}}||w_{\lambda}u||_0^2= \langle{w}_{\lambda}u
,w_{\lambda} \big(-\mathcal H^{(x)}\prx^2u -\alpha \mathcal
H^{(x)}\pry^2u 
-u^pu_x
\big)\rangle_0.
\] 
The first two terms in the sum on the right hand of the last equation
satify
\begin{align}
\langle{w}_{\lambda}u,w_{\lambda}\mathcal
H^{(x)}\partial_x^2{u}\rangle_0&=
\langle{w}_{\lambda}u,[w_{\lambda},\mathcal
H^{(x)}]\partial_x^2 u\rangle_0+ \langle{w}_{\lambda}u,\mathcal
H^{(x)} [w_{\lambda}, \partial_x^2]u\rangle_0\notag\\
\langle{w}_{\lambda}u,w_{\lambda}\mathcal H^{(x)}\partial_y^2{u}\rangle_0&=
\langle{w}_{\lambda}u,[w_{\lambda}\mathcal H^{(x)}]\partial_y^2u\rangle_0+\langle{w}_{\lambda}u,\mathcal H^{(x)}[w_{\lambda},\partial_y^2]u\rangle_0\notag
\end{align}
The  Cauchy-Schwarz inequality, the  Calderón's
commutator theorem and the lemma above  imply that
\begin{equation*}
\begin{aligned}
  \langle{w}_{\lambda}u,[w_{\lambda},\mathcal H^{(x)}]\partial_y^2u\rangle_0&\leq\|w_{\lambda}u\|_0\|[w_{\lambda},\mathcal H^{(x)}]\partial^2_yu\|_0\\
  &{\leq}{C}_1 |\partial_x w_{\lambda}| _{\infty} \|w_{\lambda}u\|_  0
  \|\partial_x^{-1}\partial_y^2u \|_0\\
  &\leq{C}_2\|w_{\lambda}u\|_0\|u\|_{{X}^s}.
\end{aligned}
\end{equation*}
On the other hand,
\begin{align*}
\langle{w}_{\lambda}u,\mathcal
H^{(x)}[w_{\lambda},\partial_y^2]u\rangle_0&\leq\|w_{\lambda}u\|_0\|[w_{\lambda},\partial_y^2]u\|_0\\
& \leq C_1\|w_{\lambda}u\|_0\left( |\partial_y^2w_{\lambda}|_\infty\|u\|_0+
  2|\partial_yw_{\lambda}|_\infty\|\partial_yu\|_0 \right)\\ &\leq C_2 \|w_{\lambda}u\|_0\|u\|_{{X}^s}.
\end{align*}
In an entirely similar way we obtain  
\[\langle{w}_{\lambda}u,[w_{\lambda},\mathcal H^{(x)}]\partial_x^2u\rangle_0\leq{C}\|w_{\lambda}u\|_0\|u\|_{{X}_s}\]
and
\[\langle{w}_{\lambda}u,\mathcal H^{(x)}[w_{\lambda}, \partial_x^2]
u\rangle_0 \leq{C}\|w_{\lambda}u \|_0\|u\|_{{X}_s}.
\] 
Also,
\[\|w_{\lambda}u^pu_x\|_0\leq|u^{p-1}u_x|_{\infty}\|w_{\lambda}u\|_0
\leq{C}_s \|w_{\lambda}u\|_0.\] With the help of the estimates above
we can infer
\[\frac{d}{dt}\|w_{\lambda}u\|_0^2\leq{A}\|u\|_{X^s}^2+B\|w_{\lambda}u\|_0^2,\]
where $A$ and $B$ are constants that do not depend on $\lambda$.
From the Gronwall inequality it is concluded that
\[\|w_{\lambda}u\|_0^2\leq{e}^{BT}(\|w_{\lambda}\phi\|_0^2+TA\|u\|_{X^s}^2).\]
Thanks to the Lebesgue's monotone convergence, it follows that
\[\|wu\|_0^2\leq{e}^{BT}(\|w\phi\|_0^2+TA\|u\|_{X^s}^2)\]
Therefore, $u(t) \in {X}_s(w^2),$ for all ${t}\in [0,T].$
By proceeding in the same way it is deduced that
\[\|w(u-v)\|_0^2\leq{e}^{BT}(\|w(\phi-\psi)\|_0^2+TA\|u-v\|_{X^s}^2),\]
where $ \psi\in X^s(w^2)$ and $v$ is the solution to \eqref{p1}, with
$\psi$ instead of $\phi$. Remains to be seen that $t\mapsto u(t)$
is continuous from $[0,T]$ in $X^s(w^2)$. But this is immediate from
dominated convergence theorem, from the continuity of $u$ in $X^s$ and from the
equation $$\|w(u(t)- u(t'))\|_{0}\le \|(w-w_\lambda)u(t)\|_{0}+
\|w_\lambda (u(t)- u(t'))\|_{0} +\|(w_\lambda-w)u(t')\|_{0}.
$$ 
This terminates the proof of the theorem.
\end{proof}
\begin{obser}
  The weights $w_{\vartheta}(x,y)=(1+x^2+y^2)^{\vartheta/2},$ for
  $\vartheta\in[0,1]$, satisfy the conditions of the  previous
  theorem.
\end{obser}

For $\gamma \ne 0$ we have the following result
\begin{thm}
  Assume that $w$ in the Theorem \ref{teorpeso1} depends only on
  $y$. Then, in this case the Cauchy problem \ref{p3} is local
  well-posed in $X^s(w^2)$.
\end{thm}
\begin{proof}
  We proceed as in the proof of Theorem \ref{teorpeso1}. Here, the
  fact that $w$ is not dependent on $y$ make our work easier. Is clear that
\[\frac{1}{2}{\frac{d}{dt}}||w_{\lambda}u||_0^2= \langle{w}_{\lambda}u
,w_{\lambda} \big(-\mathcal H^{(x)}\prx^2u -\alpha \mathcal
H^{(x)}\pry^2u +\gamma \prx^{-1}\pry ^2u -u^pu_x \big)\rangle_0.
\] 
In this case the estimates of the linear terms are
\begin{align}
\langle{w}_{\lambda}u,w_{\lambda}\mathcal
H^{(x)}\partial_x^2{u}\rangle_0&=0\notag\\
\langle{w}_{\lambda}u,w_{\lambda}\mathcal H^{(x)}\partial_y^2{u}\rangle_0&=
\langle{w}_{\lambda}u,\mathcal
H^{(x)}[w_{\lambda},\partial_y^2]u\rangle_0\notag\\
\langle{w}_{\lambda}u,w_{\lambda}\prx^{-1} \partial_y^2{u}\rangle_0&=
\langle{w}_{\lambda}u,[w_{\lambda},\partial_y^2]\prx^{-1}u\rangle_0\notag
\end{align}
Henceforth, the proof follows the same steps of the proof of Theorem
\ref{teorpeso1}.
\end{proof}
\begin{obser}
  Observe that $w(y)=y$ is a particular case of weights considered in
  the above theorem. In reality, this theorem is valid even for the
  following Cauchy problem
\begin{equation}
\left\{
\begin{aligned}
  &u_t+u^pu_x+\delta \partial_x^3u+ \mathcal H\partial_x^2u+ \alpha
  \mathcal H\partial_y^2u
  - \gamma \partial_x ^{-1}u_{yy}=0, \label{kdvbo} \\
  &u(0)=\phi,
 \end{aligned}
 \right.
\end{equation}
which represents an improvement of Theorem 2.4 in \cite{guoboling}.
\end{obser}
\section{Asymptotic behaviour of solutions with small initial data}\label{sec5}
For $\gamma=0$, in this section we show that the solution to
\eqref{BOeq} (in other words, the solution to \eqref{p1}) is global if
it is taken an small enough initial data, in a sense which will be
made precise later on. Also we show that the solution, at a time
sufficiently large, behaves as the solution to the linear equation
associated. These last are often called \emph{scattering states}.\par
For $\phi\in H^s(\R^2)$, let $P(-t)\phi= e^{-t\mathcal A}\phi$
($\mathcal A$ as in the proof of Theorem \ref{tp1}, with $\gamma=0$)
the solution to the linear problem associated to the Cauchy problem
\eqref{p1}, i.e., if $u(t)=P(-t)\phi$, $u$ satisfy the equation
$$
\frac{du}{dt}+ \mathcal H\prx^²u +\alpha \mathcal H\pry^²u =0.
$$ 
Without loss generality we can assume $\alpha=1$.\par
If $\phi \in \SSh$ then
\[
P(-t)\phi(x,y)= \frac{1}{2\pi}\int e^{i(\sgn
  (\xi)(\xi^2+\eta^2)t+x\xi+y\eta)} \widehat{\phi}(\xi,\eta)d\xi
d\eta=
\frac1{2\pi} I(t)\ast \phi(x,y)\]
where $I(t)=(e^{i\sgn (\xi)(\xi^2+\eta^2)t})^{\vee}$.
\begin{lem}
For any $x,\ y$~and $t\ne0$ real numbers,
$$I(t)(x,y)= \frac ct e^{-\frac {i}{4t}{(x^²+y^2)}}\int_{\frac x{\sqrt
    t}}^\infty e^{\frac i4
  s^2} ds + \frac{\bar c}t e^{\frac {i}{4t}{(x^²+y^2)}}\int^{\infty}_{\frac x{\sqrt{t}}} e^{-\frac i4
  s^2}ds,$$ where $c= (1+i)/2$.
\end{lem}
\begin{proof} Is clear that
$$ 
\begin{aligned}
2\pi I(1)(x,y) = &\int_\re \int_0^\infty e^{i(\xi^2+\eta^2 +x\xi
    +y\eta)}\, d\xi d\eta+ \int_\re \int_{-\infty}^0 e^{i(-\xi^2-\eta^2 +x\xi
    +y\eta)}\, d\xi d\eta\\ 
=& e^{-\frac i4(x^²+y^2)}\left( \int_\re e^{i(\eta +y/2)^2}\,
d\eta\right)\left( \int_0^\infty e^{i(\xi +x/2)^2}\, d\xi\right) +\\  &+ e^{\frac i4(x^²+y^2)}\left( \int_\re e^{-i(\eta -y/2)^2}\,
d\eta\right)\left( \int^0_{-\infty} e^{-i(\xi -x/2)^2}\, d\xi\right). 
\end{aligned}
$$
A simple change of variable prove the theorem for $t=1$. Using
the homogeneity property of the Fourier transform, the theorem follows
for any $t\ne 0$.
\end{proof}
The last lemma implies the following $L^p-L^q$ estimate for the
group $P(t)$.
\begin{prop}\label{prop4.2n}
For any $f\in L^1\bigcap L^2$, it has that 
\[|P(-t)f|_{\frac{2}{1-\theta}}\leq c|t|^{-\theta}|f|_{\frac{2}{1+\theta}},\]
for $\theta \in[0,1]$
\end{prop}

\begin{proof}
  We obtain the result by using the Young's inequality for
  convolution, the lemma above and interpolation.
\end{proof}

From Sobolev imbedding theorem it follows the following proposition. 
\begin{prop}\label{prop4.3n}
 For $s>1$ and $f\in L^1\bigcap H^s$, we have
 \[|P(-t)f|_{\infty}\leq c(1+|t|)^{-1}(|f|_1+\|f\|_s)\]
\end{prop}
Now, we are ready to prove the following theorem.
\begin{thm}\label{pp}
Let $p\geq 3$ and $s>3$. Then, there exist $\delta>0$ and $R=R(\delta)>0$
such that if $\phi \in L^1_1\bigcap H^s$ satisfies
\[|\phi|_{1,1}+\|\phi\|_s<\delta,\]
the solution $u$ to \eqref{p1} belongs to $C_b(\R,H^s)$ and satisfies
\begin{equation}
\sup_{t\in R}(1+|t|)|u(t)|_{1,\infty}\leq R.\label{ji2}
\end{equation}
\end{thm}
\begin{proof}
  For this proof we need the following lemma whose proof can be found
  in \cite{aniuraphdt} (Lemma 3.0.52)
\begin{lem}\label{lem4.4n}
For $t\geq0$, let
$$J(t)=(1+t)\int_0^t\frac{1}{(1+t-\tau)(1+\tau)^{p-1}}d\tau.$$
Then,
\begin{enumerate}
\item  $J(t)=O(1)$ as $t\to \infty$, if $p\geq3$
\item $J(t)\to \infty$ as $t\to \infty$, if $p=1,2$.
\end{enumerate}
\end{lem}
Let us see first
\begin{equation}\label{eqarriba5.1}
\| u(t)\|_s \le \|\phi\|_s \exp\left(
c\int_0^t|u_x|_\infty|u|_\infty^{p-1}d\tau\right). 
\end{equation}
Making the inner product in $H^s$ by $u$ in both sides of the equation
we obtain $$ \frac d{dt}\|u\|^2_s= -2\langle u,
u^pu_x\rangle_s. $$ By virtue of the Kato-Ponce inequality and its 
corollary (Corollary \ref{needglob}), we get
$$ \frac d{dt}\|u\|^2_s\le C|u|^{p-1}_\infty \|u\|^2_s. $$ The
inequality \eqref{eqarriba5.1} follows from this last and the Gronwall
inequality.\par Now, in light of \eqref{eqarriba5.1} it is enough
to prove \eqref{ji2}. Indeed, from the hypoteses, we have
$$ \int_0^t|u_x|_\infty|u|_\infty^{p-1}d\tau\leq
\int_0^t|u|_{1,\infty}^pd\tau \leq R^p\int_0^t(1+|\tau|)^{-p}d\tau\leq
C.$$

So let us prove \eqref{ji2}. We take $T\in(0,T_s)$ and let
$K(T)=\sup_{t\in [0,T]}\{(1+|t|)|u(t)|_{1,\infty}\}$.  From the Proposition
\ref{prop4.3n}, the Lemma \ref{lem4.4n}, \eqref{eqarriba5.1} and the integral
equation \eqref{eq2.22},
we obtain
$$
\begin{aligned}
(1+t)|u(t)|_{1,\infty}&\leq c\delta+c(1+t)
  \int_0^t(1+t-\tau)^{-1}|u(\tau)|_\infty^{p-1}\|u(\tau)\|_s^2\, d\tau
  \\ &\leq
  c\delta+c\delta^2K(T)^{p-1}e^{cK(t)^p},
\end{aligned}
$$
for $t\in [0,T].$

We choose $\delta>0$, small enough, such that the function $x\mapsto
c\delta+c\delta^2x^{p-1}e^{cx^p}-x,$ has a positive zero.  Let
$R=R(\delta)$ the first positive zero of this function. Then, the
estimates shown above imply that $K(T)\leq R$. From the fact that the
set of solutions is invariant under transformation $(t,x,y)\rightarrow
(-t,-x,-y)$ and using an extension argument the theorem is obtained.
\end{proof}
As corollary one has the following interesting theorem.
\begin{thm}
Under the hypotheses of the preceding theorem, there exists
$\phi_{\pm}\in H^s$ such that
$$\|u(t)-P(-t)\phi_{\pm}\|_r\to 0,$$
as $t\to \pm\infty$, for $r\in [s-1,s)$.
\end{thm}
\begin{proof}
See \cite{aniuraphdt}.
\end{proof}
\bibliographystyle{acm} \bibliography{mybiblio}
\end{document}